\documentclass[11pt,letterpaper]{amsart}

\usepackage{csquotes}
\usepackage{dynkin-diagrams}
\usepackage{amssymb}
\usepackage{xcolor}
\usepackage{amsmath}
\usepackage{amsthm}
\usepackage{bbm}
\usepackage{MnSymbol}
\usepackage{mathrsfs}

\usepackage{mathtools}
\usepackage[english]{babel}
\usepackage{tikz-cd}
\usepackage{hyperref}
\hypersetup{
    colorlinks=true,
    allcolors=blue
}

\def\R{\mathbbm{R}}
\def\H{\mathbbm{H}}

\def\C{\mathbbm{C}}

\def\Z{\mathbbm{Z}}

\def\P{\mathbbm{P}}

\def\L{\mathscr{L}}

\def\Pic{\text{Pic}}
\def\NS{N^1_\R}
\def\Nef{\text{Nef}}

\def\GL{\text{GL}}
\def\Eff{\text{Eff}}
\newtheorem{theorem}{Theorem}[section]
\newtheorem{prop}[theorem]{Proposition}

\newtheorem{defi}[theorem]{Definition}
\newtheorem{remark}[theorem]{Remark}
\newtheorem{conj}[theorem]{Conjecture}
\newtheorem{lemma}[theorem]{Lemma}
\newtheorem{claim}[theorem]{Claim}
\newtheorem*{question}{Question}

\title{On the nef cones of blowups of the projective plane}
\author{Luíze D'Urso}
\date{}
\address{\noindent Luíze D'Urso: IMPA, Estrada Dona Castorina 110, Rio de
  Janeiro, 22460-320, Brazil} 

\email{luize.vianna@impa.br}

\begin{document}

\maketitle
\begin{abstract}
In this paper, we study the Cremona action on the nef cone of $S_n$, the blowup of $\P^2$ in $n$ very general points, with $n\ge9$.
We construct and describe a rational polyhedral fundamental domain of the nef cone for $n=9$ with respect to this action. In the case $n\ge10$, we give a rational polyhedral fundamental domain of the $K_{S_n}$-negative part of the nef cone with respect to the Cremona action. 
\end{abstract}

\tableofcontents

\section{Introduction}


The minimal model program proposes the study of the convex geometry of nef cones in order to better understand the birational geometry of projective varieties. The nef cone $\Nef(X)$ of a normal projective variety $X$ is the dual cone of the Mori cone of $X$, which is the closure of the convex cone generated by all classes of curves in $X$. The nef cone is a closed convex cone in the space of Cartier $\R$-divisor classes modulo numerical equivalence, which is isomorphic to $\R^{\rho}$, where $\rho$ is the Picard number of $X$. The nef effective cone of $X$, $\Nef^\text{e}(X)$, is the subcone of $\Nef(X)$ of nef effective $\R$-divisors (see Definition \ref{nef efetivo}).

For Fano varieties, the nef cone coincides with the nef effective cone and is rational polyhedral, that is, it is a convex cone generated by finitely many divisor classes. Consequently, we have a very concrete way to describe the nef cone and its faces. For other varieties, the nef cone may have infinitely many generators as a convex cone, it may be round, and often presents a mix of shapes. 

However, even in some cases where the nef cone is hard to describe, it may admit a rational polyhedral cone as a fundamental domain with respect to an appropriate action. Roughly speaking, a fundamental domain is a region $R$ that contains at least one point of each orbit of the action, while $\text{int}(R)$ contains at most one point of each orbit. In this sense, $R$ may be understood as a geometric realization of representatives of the orbits. When $\Nef(X)$ is not rational polyhedral, but admits a fundamental domain, which is a rational polyhedral cone, the complexity of the geometry of the nef cone comes from the group action.

For instance, given a normal projective variety $X$, we may consider the action of the automorphism group of $X$ on divisor classes. Since automorphisms of $X$ preserve nef classes, we can obtain an action of the automorphism group of $X$ on the nef cone. Even if $\Nef(X)$ is a not polyhedral cone, this action may admit a rational polyhedral fundamental cone.

\begin{conj}[Kawamata-Morrison Cone Conjecture]
\label{conjectura}
    Let $X$ be a Calabi-Yau manifold. Then its nef effective cone admits a rational polyhedral fundamental domain with respect to the action of $\text{Aut}(X)$.
\end{conj}

References for Conjecture \ref{conjectura} are  \cite{Kawamata-1997},\cite{Morrison_1993} and \cite{morrison_1994}. When $X$ has dimension $2$, the conjecture is proved by analyzing separately different cases \cite{Sterk1985} \cite{Namikawa1985} and \cite{Kawamata-1997}. Totaro generalized the conjecture to klt Calabi-Yau pairs $(X,\Delta)$ \cite{Totaro_2008}, and also proved it in the case of surfaces \cite{Totaro-2010}.

\begin{theorem}[Totaro]
    Let $(S,\Delta)$ be a complex klt Calabi-Yau pair of dimension $2$. Then the action of $\text{Aut}(S,\Delta)$ on the nef effective cone of $S$ admits a rational polyhedral fundamental domain.
\end{theorem}

The klt restriction is in fact necessary for the result to be true, as pointed out by Totaro. In fact, there is an easy counterexample for the conjecture when the singularity of the pair is strictly log canonical. Namely, let $S$ be the blowup of $\P^2$ in $9$ very general points and $\Delta$ the strict transform of a cubic through the $9$ points. We say that a set of points is very general if no $3$ of them lie on the same line, and the same is true after a finite sequence of Cremona transformations centered at $3$ of them. Then $(S,\Delta)$ is Calabi-Yau log-canonical, but not klt. 

By \cite{Grassi93}, if the $9$ points are the base points of a pencil of cubics, then $\text{Aut}(S)$ is infinite and $\Nef^e(S)$ admits a rational polyhedral fundamental cone for the action of $\text{Aut}(S)$. However, in general, $\text{Aut}(S)$ is trivial, while $\Nef(S)$ is never polyhedral, therefore there is no polyhedral fundamental cone with respect to the action of $\text{Aut}(X)$. By allowing a different group action, namely, the so called Cremona action (see Definition \ref{cremona}), and using tools from Geometric Group Theory, we will prove the following theorem:

\begin{theorem}
\label{main theo n=9 part 1}
    The nef cone of the blowup of $\P^2$ in $9$ very general points admits a rational polyhedral cone $\mathcal{C}$ as a fundamental domain with respect to the Cremona action. The cone $\mathcal{C}$ has exactly $10$ facets and can be generated as a convex cone by $10$ Cartier divisor classes.
\end{theorem} 

\begin{remark}
    The second statement implies that the cone $\mathcal{C}$ couldn't be simpler to describe, as these are the minimum numbers of facets and extremal rays for a full dimensional cone in $\R^{10}$ not to contain a line.
\end{remark}

The Cremona action in the space of divisor classes is induced by Cremona transformations of the projective plane centered in $3$ of the blown up points. This action preserves the nef cone. To prove Theorem \ref{main theo n=9 part 1}, we make a tour through the theory of reflection groups in the $9$-dimensional hyperbolic space $\H^9$, which for now is a hyperboloid sheet in $\R^{10}$. The Cremona action also preserves this hyperboloid sheet, so it makes sense to consider this action on $\H^9$, manifesting itself as a reflection group. 

It happens that each ray of the nef cone $\Nef(S)$ intersects the hyperbolic space $\H^9$ in exactly one point, or else asymptotically. We can restrict ourselves to the study of the nonassimptotical intersection, which we denote by $\beta_9=\Nef(S)\cap\H^9\subset{\H^9}\subset\P(\R^{10})$, finding a fundamental domain ${\mathcal{P}}\subset{\H^9}$ with the aid of the theory of reflection groups. We conclude that ${\mathcal{P}}$ is the convex hull of $10$ points in $\overline{\H^9}$, therefore $\mathcal{C}$ will be the the convex cone generated by the corresponding $10$ rays in $\R^{10}$.

The fundamental domain ${\mathcal{P}}$ of $\beta_9$ is described in Proposition \ref{poliedral} and is actually very special. Once it is the convex hull of $10$ points, we call it a simplex, the simplest polytope in $\H^9$. Even more is true: its angles are all of the form $\frac{\pi}{m}$ for some integers $m$, which means $\mathcal{P}$ can tile $\H^9$ by reflections. That is, besides being a piece of puzzle of $\gamma_9$, we can see it also as a piece of puzzle of $\H^9$ if we keep reflecting it. These features are condensed in the theorem below.

\begin{theorem}
\label{main theo n=9 part 2}
    The polytope $\mathcal{P}=\mathcal{P}_9$ described in Proposition \ref{poliedral} is a hyperbolic Coxeter $9$-simplex.
\end{theorem}

In \cite{GeometryII} part II chapter $5$ section 2.3 Table $4$, we can check the classification of hyperbolic Coxeter $n$-simplices. There are three of them for $n=9$, up to congruence, and there is no Coxeter $n$-simplex for $n>9$, so it is a special object in Geometric Group Theory. Its Coxeter diagram is the following:

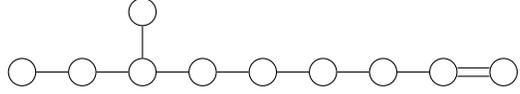
\begin{center}

\usetikzlibrary{positioning}

\begin{tikzpicture}
  [scale=.8,auto=left,every node/.style={circle,draw=black,fill=none}]
  
  \node (n1) at (1,10) {};
  \node (n2) at (2,10) {};
  \node (n3) at (3,10) {};
  \node (n4) at (4,10) {};
  \node (n5) at (5,10) {};
  \node (n6) at (6,10) {};
  \node (n7) at (7,10) {};
  \node (n8) at (8,10) {};
  \node (n9) at (9,10) {};
  \node (n10) at (3,11) {};  

  \foreach \from/\to in {n1/n2,n2/n3,n3/n4,n4/n5,n5/n6,n6/n7,n7/n8,n3/n10}
    \draw (\from) -- (\to);
    
  \draw ([yshift=-2pt]n8.east) -- ([yshift=-2pt]n9.west);
  \draw ([yshift=2pt]n8.east) -- ([yshift=2pt]n9.west);

\end{tikzpicture}.
\end{center}

In connection to the cone conjecture, Theorem \ref{main theo n=9 part 1} leads to the following question.

\begin{question}
    Does the nef effective cone of any Calabi-Yau log-canonical pair $(S,\Delta)$ admit a rational polyhedral fundamental domain with respect to some group action?
\end{question}

Generalizing this example, we may consider the blowup $S$ of $\P^2$ in $n$ very general points, and the corresponding Cremona action. In this case, there is no Coxeter $n$-simplex in $\H^n$, but we prove the following:

\begin{theorem}
\label{main theo n>9}
    Let $S$ be the blow up of $\P^2$ in $n>9$ very general points. Then the cone $$\Nef^\text{e}(S)_{K_S\le0}:=\{\alpha\in\Nef^\text{e}(S)\mid\alpha\cdot K_S\le0\}$$ admits a rational polyhedral fundamental cone with respect to the Cremona action.
\end{theorem}

\begin{remark}
    This fundamental domain will no longer be a cone over a simplex, as it has one extra facet, but we still get a concrete description, with $n+2$ facets.
\end{remark}


This paper is organized as follows. In section $2$ we present the necessary background from Geometric Group Theory to achieve our goals. Concepts concerning fundamental domains, the $n$-dimensional hyperbolic space, convex polytopes, Coxeter diagrams and reflection groups can be found in this section. In section $3$ we recall what is known about nef cones of blowups of the projective plane in $n$ very general points, and present the Cremona action. In section $4$ we relate the two previous sections, mixing Algebraic Geometry with Geometric Group Theory. At the end of the section, we present a complete description for a fundamental domain of the \textquote{conjectured nef cone} of the blowup, which coincides with the nef cone for $n\le9$. In section $5$ we study the case $n=9$, proving that this fundamental domain is in fact a rational polyhedral cone, completing the proofs of Theorems \ref{main theo n=9 part 1} and \ref{main theo n=9 part 2}. In section $6$ we study the case $n\ge10$, showing that the $K_S$-negative part of the nef cone admits a rational polyhedral fundamental cone as stated in Theorem $\ref{main theo n>9}$.

In this text, we always work over $\C$.

\section{Reflection Groups in $\H^n$ and their Fundamental Polytopes}

In this section we give a brief overview of the tools we need from Geometric Group Theory. Most of its content can be found in \cite{Dolgachev-2007}. We define the hyperbolic space $\H^n$, introduce the notion of polytopes and reflection groups in $\H^n$, and present Theorem \ref{teorema1}, which gives conditions for a polytope to be a fundamental domain with respect to a reflection group. We start by recalling the definition of a fundamental domain. By $X^{\circ}$ we denote the interior of a topological space $X$.

\begin{defi}
\label{dom fund}
    Let $G$ be a group that acts faithfully on a topological space $T$ by homeomorphisms. A fundamental domain of $T$ with respect to this action is a closed region $R\subset T$ such that $T=\bigcup_{g\in G} gR$ and $R^{\circ}\cap(gR)^{\circ}=\emptyset$ for any $g\in G\backslash \{e\}.$
\end{defi} 

In this section, the topological space $T$ is $\H^n$ while $G$ is a reflection group (see Definition \ref{reflection}). In this case, the fundamental domain is a convex polytope in $\H^n$ (see Definition \ref{polytope}), and we call it a fundamental polytope. Later on, we will take different spaces to be $T$.

We define the $(n+1)$-dimensional generalized Minkovski space $\R^{1,n}$ to be the vector space $\R^{n+1}$ endowed with a bilinear form of signature $(1,n)$. We fix an orthogonal basis $\{e_0,e_1,\dots,e_n\}$ such that $e_0^2=1$ and $e_i^2=-1$ for every $i\ge1$.


\begin{defi}
\label{hyperbolic space}
    The $n$-dimensional hyperbolic space is $$\H^n:=\{v\in\R^{1,n}\mid v^2=1,\ v\cdot e_0>0\}\subset\R^{1,n},$$ with distance defined as $\text{dist}(x,y)=\cosh^{-1}(x\cdot y)$.
\end{defi}

\begin{remark}
The hyperbolic space $\H^n$ is the unique simply connected $n$-dimensional Riemannian manifold of constant curvature $-1$, up to isometry. There are many ways to describe it as a subset of an Euclidean space together with an explicit Riemannian metric. The one in Definition \ref{hyperbolic space} is known as the hyperboloid model, or Minkowski model.
\end{remark}

The light cone is defined as $$\L=\L_n:=\{v\in\R^{1,n}\mid v^2\ge0\}\subset\R^{1,n}.$$ Notice that each point of $\H^n$ corresponds to exactly one line in the light cone, namely, the line connecting that point to the origin. In this sense, $\H^n$ may be described as the projetivization of the interior of $\L_n$, that is, $$\H^n\cong\P(\L_n^{\circ})=\frac{\L_n^{\circ}}{\sim},$$ with $\L_n^{\circ}=\{v\in\R^{1,n}\mid v^2>0\}$ and the equivalence relation given by nonzero scalar multiplication. It is also a convention to interpret $\partial\H^n$ as the lines through the origin that are contained in $\partial\L_n=\{v\in\R^{1,n}\mid v^2=0\}$: $$\partial\H^n\cong\P(\partial\L_n)=\frac{\partial\L_n\backslash\{(0,\dots,0)\}}{\sim}.$$ Set $$\overline{\H^n}=\H^n\cup\partial\H^n\cong\P(\L_n)=\frac{\L_n\backslash\{(0,\dots,0)\}}{\sim},$$ the compactification of $\H^n$.

From now on, we use coordinates $(x_0,x_1,\dots,x_n)$ to describe a vector in $\R^{1,n}$, and $(x_0:x_1:\dots:x_n)$ to describe the point in $\overline{\H^n}$ that corresponds to the line through the origin and the point $(x_0,x_1,\dots,x_n)\in\L_n\backslash\{(0,0,\dots,0)\}$.

\begin{defi}
\label{ort}
    Given a vector $u\in\R^{1,n}$ such that $ u^2=-2$, define the hyperplane $u^{\perp}\subset\H^n$ by $$u^{\perp}=\{x\in\H^n\mid x\cdot u=0\}.$$ We call $u$ and $-u$ the normal vectors of the hyperplane $u^{\perp}=(-u)^{\perp}$. An open halfspace in $\H^n$ associated to a normal vector $u$ is the set $$u^{>0}=\{x\in\H^n\mid x\cdot u>0\},$$ and a closed halfspace in $\H^n$ is the closure $$u^{\ge0}=\{x\in\H^n\mid x\cdot u\ge0\}$$ of an open halfspace in $\H^n$.
\end{defi}


\begin{remark}
    Once $u^2<0$, the condition $u^2=-2$ can be achieved after normalization by a positive scalar. The condition $u^2<0$ is important because it implies that the hyperplane $u^{\perp}\subset\R^{1,n}$ intersects the hyperbolic space $\H^n$, equivalently, the interior $\L_n^{\circ}$ of the light cone. In the context of vectors in $\R^{n+1}$ with the standard inner product $\left<\ ,\ \right>$, a hyperplane $v^\perp=\{x\in\R^{n+1}\mid \left<x,v\right>=0\}$ intersects $\L_n^{\circ}$ if and only if $v\notin\L_n$. In $\R^{1,n}$, the normal vectors $u$ and $-u$ to the hyperplane $v^{\perp}$ are Euclidean reflections of $v$ through the $e_0$ axis, modulo a nonzero scalar multiplication. In other words, $v\notin\L_n\iff u\notin\L_n$.
\end{remark}

\begin{defi}
    A collection of hyperplanes $\{{\lambda}^{\perp}\}_{\lambda\in\Lambda}\subset\H^n$ is said to be locally finite if each compact subset of $\H^n$ intersects only finitely many of them. A collection of halfspaces $\{{\lambda}^{\ge0}\}_{\lambda\in\Lambda}\subset\H^n$ is said to be locally finite if the corresponding collection of hyperplanes $\{{\lambda}^{\perp}\}_{\lambda\in\Lambda}\subset\H^n$ is locally finite.
\end{defi}

\begin{defi}
\label{polytope}
    A convex polytope in $\H^n$ $$\mathcal{P}=\bigcap_{i\in I} {u_i}^{\ge0}$$ is the intersection of a locally finite collection of halfspaces $\{{u_i}^{\ge0}\}_{i\in I}$. In this case, we say that $\mathcal{P}$ is determined by the collection of normal vectors $\{u_i\}_{i\in I}$.
\end{defi} 

\begin{remark}
\label{consistencia}
    In the context of Definition \ref{polytope}, we always assume that no halfspace ${u_j}^{\ge0}$ contains the intersection $\bigcap_{i\in I\backslash\{j\}} {u_i}^{\ge0}$ of the other ones, so the set of halfspaces we use to describe the polytope is minimal. 
\end{remark}

\begin{defi}
\label{simplex}
    A polytope in $\H^n$ is said to have finite volume if its closure in $\overline{\H^n}$ is the convex hull of finitely many points of $\overline{\H^n}$. It is called a simplex if it has finite volume and is determined by exactly $n+1$ halfspaces, or, equivalently, if its closure in $\overline{\H^n}$ is the convex hull of $n+1$ points of $\overline{\H^n}$.
\end{defi}

\begin{remark}
    It is possible to define the volume of a polytope in $\H^n$ since it is a metric space. This would give an equivalent notion of having finite volume as in Definition \ref{simplex}.
\end{remark}

Once we fix an order for the normal vectors to the facets of a polytope, we can associate to it a Cartan Matrix, which encodes information about the angles of the polytope.

\begin{defi}
\label{cartan}
    The angle between two halfspaces $u^{\ge0}$ and $v^{\ge0}$ with $u^2=v^2=-2$ is the real number $0\le\theta\le\pi$ such that $u\cdot v=2\cos(\theta)$, as long as $-2\le u\cdot v\le 2$. If $|u\cdot v|>2$, we say the hyperplanes $u^{\perp}$ and $v^{\perp}$ are divergent, in which case, as hyperplanes of $\R^{1,n}$, they intersect outside of the light cone. The Cartan Matrix of a polytope determined by the collection of normal vectors $\{v_i\}_{i\in I}$ is a (possibly infinite) matrix with entries $a_{ij}=-v_i\cdot v_j=-2\cos(\theta_{ij})$. 
\end{defi}

\begin{remark}
    By Definition \ref{cartan}, the angle between a halfspace $u^{\ge0}$ with itself is $\pi$.
\end{remark}

For any polytope in $\H^n$, the angles between its defining halfspaces encode the necessary information to determine whether it is a fundamental domain of a reflection group, as we see in Theorem \ref{teorema1} below.

\begin{defi}
    A convex polytope in $\H^n$ is said to be a Coxeter polytope if the angles between any two halfspaces are either $0$, a submultiple of $\pi$, that is, $\frac{\pi}{m}$ for some $m\in\Z_{>0}$, or the corresponding halfspaces are divergent. If the angle is $0$, by abuse of notation, we consider that $m=\infty$. A simplex which is a Coxeter polytope is called a Coxeter simplex.
\end{defi}

\begin{defi}
\label{reflection}
    A reflection of $\H^n$ is an isometry of order $2$, whose set of fixed points is a hyperplane of $\H^n$. A group generated by reflections of $\H^n$ is called a reflection group if its orbits in $\H^n$ are all discrete.
\end{defi}

\begin{theorem}[\cite{Dolgachev-2007}, Theorem 2.1] \label{teorema1}
    For every Coxeter polytope $\mathcal{P}$ in $\H^n$, the group $\Gamma$ generated by the reflections with respect to its facets is a reflection group, and $\mathcal{P}$ is a fundamental domain for $\H^n$ with respect to the action of $\Gamma$. Conversely, any reflection group $\Gamma$ in $\H^n$ admits a Coxeter polytope $\mathcal{P}$ as a fundamental domain.
\end{theorem}

In other words, Theorem \ref{teorema1} says that Coxeter polytopes in $\H^n$ are the polytopes that can tile $\H^n$ by reflections.

\begin{defi}
    To any Coxeter polytope we associate a graph, called Coxeter diagram, as follows. The vertices are the halfspaces that delimit the polytope. Two distinct vertices are joined with $m-2$ edges if the corresponding halfspaces make an angle of $\theta=\frac{\pi}{m}$, with $m\ge3$. They are joined with a dotted edge if the corresponding hyperplanes are divergent. If $m=2$, there is no joining edge, and if $m=\infty$ they are joined by a dashed edge.
\end{defi}

\begin{remark}
    By Remark \ref{consistencia}, if two different halfspaces that delimit a polytope make an angle of $\pi$, then one of them is irrelevant to describe the polytope, which is assumed not to happen. The angle $\pi$ will only appear as the angle of each halfspace with itself, so it has no impact on the Coxeter diagram.
\end{remark}

\section{The Cremona Action and the Nef Cone}

Let $\P^2_n$ be the blowup of $\P^2$ at $n$ points, say $P_1,P_2,\dots,P_n\in\P^2$, in general posistion, i.e. not three of them in the same line. Recall that, as a vector space, the Néron-Severi Space $$N^1_{\R}(\P^2_n):=\left(\Pic(\P^2_n)/ \equiv \right)\otimes{\R}$$ is isomorphic to $\R^{n+1}$. Moreover, the intersection product allows us to endow it with a symmetric bilinear form of signature $(1,n)$. More precisely, denote by $e_0\in\NS(\P^2_n)$ the class of the pullback of a line in $\P^2$ and $e_i\in\NS(\P^2_n)$ the class of the exceptional divisor over $P_i$ for $i=1,2,\dots,n$. With this choice of generators, we have \[\NS(\P^2_n)\cong\bigoplus_{i=0,1,\dots,n}\R e_i,\]
where $e_0^2=1$, $e_i ^2=-1$ for $i=1,2,\dots,n$ and $e_i \cdot e_j=0$ for $i\neq j$. In other words, $\NS(\P^2_n)\cong\R^{1,n}$, the generalized Minkovski space of dimension $n+1$ and we can use all concepts from the previous section. Let $x_0, x_1, \dots, x_n$ be the coordinates in $\R^{1,n}$.

For each $3$ distinct blown up points $P_i,P_j,P_k$, the standard quadratic transformation of $\P^2$ centered at these points induces a linear transformation $\varphi_{ijk}$ in $\R^{1,n}$, which preserves the bilinear form. This linear map acts on the space spanned by $\{e_0,e_i,e_j,e_k\}$ according to the matrix
\begin{equation*}
\begin{pmatrix}
2 & 1 & 1 & 1  \\
-1 & 0 & -1 & -1 \\
-1 & -1 & 0 & -1 \\
-1 & -1 & -1 & 0 
\end{pmatrix},
\end{equation*} 
and it fixes the remaining vectors of the basis $\{e_0,e_1,\dots,e_n\}$. The set of fixed points of $\varphi_{ijk}$ is the hyperplane $x_0+x_i+x_j+x_k=0$ in $\NS(\P^2_n)$. In particular, the vector $-K_{\P^2_n}=3e_0-\sum_{i=1}^n e_i=(3,-1,-1,\dots,-1)$ is a fixed point for all distinct $i,j,k$. 

\begin{defi}
\label{cremona}
    The Cremona action in $\NS(\P^2_ n)$ is given by the linear group $$W=W_n=\left<\varphi_{ijk}\right>\subset\GL(n+1,\R),$$ generated by all $n\choose3$ transformations $\varphi_{ijk}$ with distinct $i,j,k$ $\in$ $\{1,2,\dots,n\}$. 
\end{defi}

We can choose different sets of generators for $W$. For instance, the composition $\sigma_{12}:=\varphi_{134}\circ\varphi_{234}\circ\varphi_{134}$ coincides with the linear transformation in $\R^{n+1}$ that permutes the coordinates $x_1$ and $x_2$. Let $\sigma_{i,i+1}$ the linear transformation that permutes $x_i$ and $x_{i+1}$.
One can check that $$W=\left<\varphi_{123},\sigma_{1,2},\sigma_{2,3},\dots\sigma_{{n-1},n}\right>.$$ We write $w_0:=\varphi_{123}$ and $w_i:=\sigma_{i,i+1}$ to simplify the notation. 
In this way, 
\begin{equation}
\label{geradores}
W=\left<w_0,w_1,\dots,w_{n-1}\right>
\end{equation} 
is generated by $n$ elements of order $2$.

\begin{remark}
    For $n\le 8$, the group $W$ is finite, and is also known as a Weyl Group in the context of semisimple Lie algebras. For $n\ge9$, $W$ is infinite.
\end{remark}

\begin{defi}
\label{very_general_points}
Let $n$ points $P_1,P_2,\dots,P_n\in\P^2$, no three of them in a line. Take any quadratic transformation centered at three of them, say $P_1,P_2$ and $P_3$. This gives us a new set of points: $P_1,P_2,P_3$ and the images of $P_4,\dots,P_n$.
    We say that the points $P_1,P_2,\dots,P_n\in\P^2$ are \textit{very general} or \textit{in very general position} if for any iteration of this process, the new set of points does not contain three in a line. 
\end{defi}

From now on, let $S=S_n=\P^2_n$ be the blowup of $\P^2$ in $n$ very general points. 
In $\NS(S)$, we consider the following important cones. 

\begin{defi}The effective cone of $S$, \[{\Eff}(S)={\{\sum_{i=1,\dots,r} a_i c_i \mid r\in\Z_{\ge0}, a_i\in\R_{>0}, c_i\in\NS(S) \text{ is the class of a curve} \}},\] is the cone of classes of effective $\R$-divisors, which is not necessarily closed. Its closure is called the Mori cone $$\overline{\text{NE}}(S)=\overline{\Eff(S)}.$$ The nef cone of $S$ is the dual cone of $\overline{\text{NE}}(S)$, 
\[\Nef(S)=\{v\in\NS(S)\mid v\cdot c\ge0 \text{ for every curve } c\subset S\}.\] 
\end{defi}

Since $W$ acts on $\NS(S)$ sending classes of curves in classes of curves, then all these cones 
are $W$-invariants. Moreover, this action preserves the intersection form and the lattice $\Pic(S)$. The geometry of these cones is related to $-1$-curves, which are defined below.

\begin{defi}
    A $-1$-curve in $S$ is a curve $c\subset S$ isomorphic to $\P^1$, such that $c^2=-1$. By abuse of notation, we sometimes write $c$ to denote its class in $\NS(S)$.
\end{defi} 

In general, for $n>9$ the geometry of the cones defined above is not completely understood. We mention he following famous conjecture (see \cite{Tommaso} and its errata \cite{errata} for a deeper discussion).
\begin{conj}[$-1$-Curves Conjecture]
\label{conj-1}
The Mori cone of $S$ is generated by the light cone and the $-1$-curves. Dually, the nef cone of $S$ is given by \[\Nef(S)=\{v\in\L\mid v\cdot c\ge0 \text{ for every }-1\text{-curve } c\subset S\}.\]
\end{conj}


\begin{defi}
    An \textit{extremal ray} of a convex cone $\mathcal {C}$ in $\R^{1,n}$ is a ray $\mathcal{R}\subset\mathcal{C}$ with the following property: if $u,v\in\mathcal{C}$ are such that $u+v\in \mathcal{R}$, then $u,v\in \mathcal{R}$. A ray in $\NS(S)$ is said to be rational if it contains some non zero divisor class of $S$. Otherwise, we call it irrational. 
\end{defi}

\begin{defi}
    A convex cone in $\R^{1,n}$ is said to be \textit{polyhedral} if it generated by finitely many vectors, that is, if it contains finitely many extremal rays. A polyhedral cone in $\NS(S)$ is said to be \textit{rational} if all its finitely many extremal rays are rational. 
\end{defi}


When $n\le8$, the surface $S$ is Fano. By the Kleiman's ampleness criterion, this is equivalent to the anticanonical class $-K_S\in\NS(S)$ lying in the interior of the nef cone. As a consequence of \cite{BCHM} Corollary 1.3.2, every smooth Fano variety has a rational polyhedral nef cone. If $3\le n\le8$, the cone $\overline{\text{NE}}(S)$ is generated by the finitely many $-1$-curves $c$, which form an orbit of the action of $W$. Each of them defines a halfspace $\{v\in\NS(S)\mid v\cdot c\ge0\}$ in $\NS(S)$, and Conjecture \ref{conj-1} holds. 

\begin{figure}[h]
\centering

\tikzset{every picture/.style={line width=0.75pt}} 

\begin{tikzpicture}[x=0.75pt,y=0.75pt,yscale=-1,xscale=1]

\draw    (249,123) -- (343,331) ;
\draw    (481,128) -- (343,331) ;
\draw  [dash pattern={on 0.84pt off 2.51pt}]  (301,98) -- (343,331) ;
\draw  [dash pattern={on 0.84pt off 2.51pt}]  (402,95) -- (343,331) ;
\draw    (347,164) -- (343,331) ;
\draw    (276,183) -- (345,247.5) ;
\draw    (345,247.5) -- (428,205) ;
\draw  [dash pattern={on 0.84pt off 2.51pt}]  (308,142) -- (276,183) ;
\draw  [dash pattern={on 0.84pt off 2.51pt}]  (308,142) -- (391,142) ;
\draw  [dash pattern={on 0.84pt off 2.51pt}]  (391,142) -- (428,205) ;
\draw  [fill={rgb, 255:red, 74; green, 144; blue, 226 }  ,fill opacity=1 ][line width=0.75]  (318,152) .. controls (318,149.79) and (319.79,148) .. (322,148) .. controls (324.21,148) and (326,149.79) .. (326,152) .. controls (326,154.21) and (324.21,156) .. (322,156) .. controls (319.79,156) and (318,154.21) .. (318,152) -- cycle ;

\draw (319.94,155.43) node [anchor=north west][inner sep=0.75pt]  [font=\scriptsize,color={rgb, 255:red, 74; green, 144; blue, 226 }  ,opacity=1 ,rotate=-1.01]  {$-K_{S}$};

\end{tikzpicture}
\caption{A sketch of $\Nef(S)$ for $3\le n\le8$.}
\end{figure}

We are actually interested in the case $n\ge9$, when there are infinitely many $-1$-curves, and each of them generates an extremal ray, so $\overline{NE}(S)$ and $\Nef(S)$ are not polyhedral. 
When $n=9$, the extremal rays of $\overline{NE}(S)$ are generated by the infinitely many $-1$-curves, which form an orbit of the action of $W$, and also the ray spanned  by $-K_S$, which is a limit ray. Similarly, the extremal rays of $\Nef(S)$ are spanned  by the $W$-orbits of $(1,0,0,0,0,0,0,0,0,0)$ and $(1,-1,0,0,0,0,0,0,0,0)$ and by $-K_S$. The extremal rays spanned  by the two orbits accumulate in the ray spanned  by $-K_S$. In this case, $K_S^2=0$, and both the Nef cone and the Mori cone are contained in the halfspace $\{v\in\NS(S)\mid v\cdot K_S\le0\}$. Moreover, Conjecture \ref{conj-1} also holds in this case.

\begin{figure}[h]
    \centering

\tikzset{every picture/.style={line width=0.75pt}} 

\begin{tikzpicture}[x=0.75pt,y=0.75pt,yscale=-1,xscale=1]

\draw [color={rgb, 255:red, 74; green, 144; blue, 226 }  ,draw opacity=1 ]   (222,86) -- (328,248) ;
\draw    (384,101) -- (328,248) ;
\draw    (307,119) -- (328,248) ;
\draw    (272,121) -- (328,248) ;
\draw    (250,112) -- (328,248) ;
\draw    (238,102) -- (328,248) ;
\draw  [dash pattern={on 0.84pt off 2.51pt}]  (343,57) -- (328,248) ;
\draw  [dash pattern={on 0.84pt off 2.51pt}]  (263,57) -- (328,248) ;
\draw  [dash pattern={on 0.84pt off 2.51pt}]  (236,69) -- (328,248) ;
\draw  [dash pattern={on 0.84pt off 2.51pt}]  (227,77) -- (328,248) ;
\draw    (260,133) -- (282,143) ;
\draw    (253,127) -- (260,133) ;
\draw    (282,143) -- (312,147) ;
\draw    (312,147) -- (376,124) ;
\draw  [dash pattern={on 0.84pt off 2.51pt}]  (341,79) -- (376,124) ;
\draw  [dash pattern={on 0.84pt off 2.51pt}]  (269,75) -- (341,79) ;
\draw  [dash pattern={on 0.84pt off 2.51pt}]  (244,85) -- (269,75) ;
\draw  [dash pattern={on 0.84pt off 2.51pt}]  (236,93) -- (244,85) ;
\draw  [fill={rgb, 255:red, 74; green, 144; blue, 226 }  ,fill opacity=1 ] (233,111) .. controls (233,108.79) and (234.79,107) .. (237,107) .. controls (239.21,107) and (241,108.79) .. (241,111) .. controls (241,113.21) and (239.21,115) .. (237,115) .. controls (234.79,115) and (233,113.21) .. (233,111) -- cycle ;

\draw (215,119.4) node [anchor=north west][inner sep=0.75pt]  [font=\footnotesize]  {$\textcolor[rgb]{0.29,0.56,0.89}{-K}\textcolor[rgb]{0.29,0.56,0.89}{_{S}}$};

\end{tikzpicture}
\caption{A sketch of $\Nef(S)$ for $n=9.$}

\end{figure}

When $n\ge10$, Conjecture \ref{conj-1} is not known to hold. What we do know is that if $\{\mathcal{R}_i\}_{i\in I}$ are all the rays spanned  by $-1$-curves $c_i$, then the Mori cone $\overline{NE}(S)$ contains the cone generated by the light cone and all $\mathcal{R}_i$, while Conjecture \ref{conj-1} states equality. In addition, we know that any $-1$-curve $c_i$ satisfies $c_i\cdot K_S=-1<0$. If there is any other extremal ray $\mathcal{R}$ of $\overline{NE}(S)$, spanned  by a vector $v$ outside the light cone, it satisfies $v\cdot K_S>0$. 

\begin{remark}
\label{lemma4.1_Tommaso}
    By \cite{errata}, Lemma 4.1, $$\overline{NE}(S)\subset\overline{\mathcal{R}(-K_S)+ \sum\mathcal{R}_i}.$$ Taking the dual cones, we get that $$\Nef(S)\supset\{v\mid v\cdot K_S\le0 \text{ and } v\cdot c_i\ge0\ \forall i\}$$ and conclude by the definition of $\Nef(S)$ that $$\{v\in\Nef(S)\mid v\cdot K_S\le0\}=\{v\mid v\cdot K_S\le0 \text{ and } v\cdot c_i\ge0\ \forall i\}.$$ 
\end{remark}

\begin{figure}[h]

\tikzset{every picture/.style={line width=0.75pt}} 

\begin{tikzpicture}[x=0.75pt,y=0.75pt,yscale=-1,xscale=1]

\draw  [fill={rgb, 255:red, 74; green, 144; blue, 226 }  ,fill opacity=1 ] (98,66.5) .. controls (98,64.57) and (99.57,63) .. (101.5,63) .. controls (103.43,63) and (105,64.57) .. (105,66.5) .. controls (105,68.43) and (103.43,70) .. (101.5,70) .. controls (99.57,70) and (98,68.43) .. (98,66.5) -- cycle ;
\draw [color={rgb, 255:red, 74; green, 144; blue, 226 }  ,draw opacity=1 ]   (309.04,30.82) -- (245.05,134.58) ;
\draw    (180,100) -- (309,274) ;
\draw [color={rgb, 255:red, 74; green, 144; blue, 226 }  ,draw opacity=1 ]   (245.05,134.58) -- (309,274) ;
\draw [color={rgb, 255:red, 74; green, 144; blue, 226 }  ,draw opacity=1 ] [dash pattern={on 0.84pt off 2.51pt}]  (309.04,30.82) -- (309,274) ;
\draw    (375,70) -- (309,274) ;
\draw    (337,119) -- (309,274) ;
\draw  [dash pattern={on 0.84pt off 2.51pt}]  (362,41) -- (309,274) ;
\draw    (299,132) -- (309,274) ;
\draw    (268,135) -- (309,274) ;
\draw  [dash pattern={on 0.84pt off 2.51pt}]  (318,33) -- (309,274) ;
\draw    (299,132) -- (268,135) ;
\draw    (337,119) -- (299,132) ;
\draw    (375,70) -- (337,119) ;
\draw    (362,41) -- (375,70) ;
\draw    (318,33) -- (362,41) ;
\draw    (252,134) -- (309,274) ;
\draw    (268,135) -- (252,134) ;
\draw    (309.04,30.82) .. controls (217,17) and (96,92) .. (245.05,134.58) ;

\draw (103.5,73.4) node [anchor=north west][inner sep=0.75pt]    {$\textcolor[rgb]{0.29,0.56,0.89}{-K_{S}}$};

\end{tikzpicture}

\caption{A sketch of $\Nef(S)$ for $n=10$ if conjecture \ref{conj-1} holds.}
\end{figure}

\begin{defi}
\label{nef efetivo}The nef effective cone of $S$ is the intersection $$\Nef^{\text{e}}(S)=\Nef(S)\cap\Eff(S).$$
\end{defi} 
\begin{remark}
It is well known that $$\Nef(S)\subset\L\subset\overline{\text{NE}}(S).$$ In fact, $$\Nef(S)\backslash\Nef^\text{e}(S)\subset\partial\L$$
is the union of the irrational rays of $\Nef(S)\cap\partial\L$.

For $n\le9$, $$\Nef(S_n)=\Nef^\text{e}(S_n).$$ 
\end{remark}

\begin{defi} The $K_S$-negative side of $\Nef^\text{e}(S)$ is
\begin{align*}
    \Nef^\text{e}(S)_{K_S\le0}:&=\{v\in\Nef^\text{e}(S)\mid v\cdot K_S\le0\}\\ &=\{v\mid v\cdot K_S\le0 \text{ and } v\cdot c_i\ge0\ \forall i\}\cap\Eff(S).
\end{align*}
\end{defi}

\section{The Cremona Action on $\H^n$}


In this section, we interpret the Cremona action as an action by reflections in $\H^n$, and the projectivisation of the nef cone as a subset of $\overline{\H^n}$. This bridge is what we need to find the desired fundamental cone. 
Recall that $S=S_n=\P^2_n$ is the blowup of $\P^2$ in $n$ very general points, as in Definition \ref{very_general_points}. From now on, we always refer to $\NS(S)$ as $\R^{1,n}$ with the base described in the previous section.

Consider the Cremona action of $W=\left<w_0,w_1,\dots w_{n-1}\right>$ in $\R^{1,n}$ from Definition \ref{cremona}. Recall that the linear maps $w_i$ preserve the product in $\R^{1,n}$. In particular, they preserve the hyperboloid $\{v\in\R^{1,n}\mid v^2=1\}$, which has two connected components, so each of these maps either preserves or switches them. 
Since $e_0$ and $w_i(e_0)$ lie on the same component for any $i=0, 1, \dots, n-1$, all these linear transformations of $\R^{1,n}$ preserve the hyperbolic space $\H^{n}$ and its metric, as in Definition \ref{hyperbolic space}.

The restrictions of these generators of $W$ to $\H^n$ are reflections, as introduced in Definition \ref{reflection}. The set of fixed points of $w_0$ is given by $x_0+x_1+x_2+x_3=0$, which is the orthogonal hyperplane to the normal vectors $\pm(e_0-e_1-e_2-e_3)$ in $\H^n$. The set of fixed points of $w_i$ for $i=1,\dots,n-1$ is given by $x_i-x_{i+1}=0$, which is the orthogonal hyperplane to the normal vectors $\pm(e_i-e_{i+1})$ in $\H^n$. 

\begin{theorem}
The polytope in $\H^n$ 
\begin{equation}
\label{p~}
    \Tilde{\mathcal{P}}=\Tilde{\mathcal{P}}_n:=\{x_0\ge -x_1-x_2-x_3\text{ and }x_1\le x_2\le \dots\le x_n\}
\end{equation} 
is a fundamental domain of $\H^n$ with respect to the action of $W$.
\end{theorem} 

\begin{proof}
The inequalities $x_0\ge -x_1-x_2-x_3$ and $x_1\le x_2\le \dots\le x_n$ defining the polytope $\Tilde{\mathcal{P}}$ correspond to the halfspaces $(v_0=e_0-e_1-e_2-e_3)^{\ge0}$ and $(v_i=e_{i}-e_{i+1})^{\ge0}$, $i=1,\dots,n-1$, respectively. The Cartan Matrix of this polytope $C_{\Tilde{\mathcal{P}}}=[a_{ij}=-v_i\cdot v_j=-2\cos(\theta_{ij})]$ is 

\begin{equation*}
C_{\Tilde{\mathcal{P}}}=
\begin{pmatrix}
2 & 0 & 0 & -1 & 0 &  \dots & 0 & 0 \\
0 & 2 & -1 & 0 & 0 &  \dots & 0 & 0 \\
0 & -1 & 2 & -1 & 0 &  \ddots & \vdots & \vdots \\
-1 & 0 & -1 & 2 & -1 &  \ddots & 0 & 0 \\
0 & 0 & 0 & -1 & 2 &  \ddots & 0 & 0 \\
 \vdots & \vdots  &  \ddots & \ddots  & \ddots  & \ddots & -1 & 0  \\
0 & 0 & \dots & 0 & 0 & -1 & 2 & -1  \\
0 & 0 & \dots & 0  & 0 & 0  & -1 & 2  \\
\end{pmatrix}.
\end{equation*}

All the entries are $2$, $0$ and $-1$, and the angles associated to the corresponding pair of halfspaces are $\pi$, $\frac\pi2$ and $\frac\pi3$ respectively, as in Definition \ref{cartan}. The angles are all submultiples of $\pi$, therefore $\Tilde{\mathcal{P}}$ is a Coxeter polytope. By Theorem \ref{teorema1}, $\Tilde{\mathcal{P}}$ is a fundamental domain of $\H^n$ with respect to the action of $W$. 

  


\end{proof}


\begin{defi}
We define 
the polytope
$$\beta=\beta_n=\bigcap_{-1\text{-curve}\ c}(\sqrt{2}c)^{\ge0}\subset\H^n.$$ If Conjecture \ref{conj-1} holds, then $$\beta=\Nef^\text{e}(S)\cap\H^n=\Nef(S)\cap\H^n\subset\H^n.$$

\end{defi}

\begin{remark}
The $\sqrt2$ appears in the formula so that the normal vector satisfies $(\sqrt2c)^2=-2$, as required in Definition \ref{ort}.
\end{remark}

 Finding a fundamental domain $\mathcal{C}$ for $\Nef^\text{e}(S)$ with respect to the action of $W$ in $\R^{1,n}$ is related to finding a fundamental domain $\mathcal{P}=\mathcal{C}\cap\H^n$ for $\beta_n$ with respect to the action of $W$ in $\H^n$. Moreover, $\mathcal{C}$ will be a polyhedral cone if and only if $\mathcal{P}$ has finite volume.

\begin{lemma}
\label{lema}
    If $\Tilde{\mathcal{P}}_n$ is a fundamental domain of $\H^n$ with respect to $W$, then the intersection $\mathcal{P}=\mathcal{P}_n=\Tilde{\mathcal{P}}_n\cap\beta_n$ is a fundamental domain of $\beta_n$ with respect to $W$. 
\end{lemma}

\begin{proof}
    Since $\mathcal{P}\subset\Tilde{\mathcal{P}}$, and $\Tilde{\mathcal{P}}$ is a fundamental domain of $\H^n$, we get $\mathcal{P}^{\circ}\cap w\mathcal{P}^{\circ}\subset \Tilde{\mathcal{P}}^{\circ}\cap w \Tilde{\mathcal{P}}^{\circ}=\emptyset$ for all nontrivial $w\in W$. Moreover, $W \mathcal{P}$ covers $\beta$, because for all $x\in\beta$ there is $w\in W$ and $y\in\Tilde{\mathcal{P}}$ such that $x=wy$. And since $W$ preserves $\beta$, then $y=w^{-1}x\in\beta$ and therefore $y\in\mathcal{P}$. 
\end{proof} 

\begin{theorem}
\label{o dominio fundamental}
    The polytope $\mathcal{P}=\mathcal{P}_n=\Tilde{\mathcal{P}}_n\cap\beta_n$ in $\H^n$ is given by the $n+1$ inequalities $x_0\ge -x_1-x_2-x_3$ and $x_1\le x_2\le \dots\le x_n\le 0$. In other words, $\mathcal{P}=\Tilde{\mathcal{P}}\cap(x_n\le0)$.
\end{theorem}

\begin{proof}
It is easy to check that $\mathcal{P}\subset \Tilde{\mathcal{P}}\cap (x_n\le0)$, since $x_n\le0$ is the halfspace $(\sqrt{2}e_n)^{\ge0}$. To show that equality holds, we prove that $\Tilde{\mathcal{P}}\cap (x_n\le0)\subset\beta_n$. It is enough to show that any other inequality corresponding to a halfspace $(\sqrt{2}c)^{\ge0}$ for a $-1$-curve $c$ is redundant when added to the ones we have: $x_0+x_1+x_2+x_3\ge0$ and $x_1\le x_2\le \dots\le x_n\le0$. 

The proof is by induction on the degree $d$ of the $-1$-curve associated to the halfspace. Recall that the degree $d$ of a curve $c$ is the product $d=c\cdot e_0$. In other words, it is its first coordinate in $\R^{n+1}$. We first analyze the base cases when $d=0$ and $d=1$.

The $-1$-curves of degree $0$, that is, of the form $e_i$ with $i=1,\dots,n$, correspond to the halfspaces $x_i\le0$. Notice that the region $\Tilde{\mathcal{P}}\cap (x_n\le0)$ is already contained in this halfspace, because $x_1\le x_2\le \dots\le x_n\le0$.

The $-1$-curves of degree $1$, that is, of the form $e_0-e_i-e_j$ for distinct $i,j=1,\dots,n$, correspond to the halfspaces $x_0+x_i+x_j\ge0$. Assume without loss of generality that  $i<j$. Then in $\Tilde{\mathcal{P}}\cap (x_n\le0)$ we have $x_0\ge x_0$, $x_i\ge x_1$, $x_j\ge x_2$ and $0\ge x_3$. Adding up these inequalities, we get $x_0+x_i+x_j\ge x_0+x_1+x_2+x_3\ge0$. So, the inequality $x_0+x_i+x_j\ge0$ is redundant in $\Tilde{\mathcal{P}}\cap (x_n\le0)$.

Now that the base case is complete, take $c=(d,-m_1,\dots,-m_n)$ the class of a $-1$-curve of degree $d>1$. The values $d,m_1,\dots,m_n$ must be nonnegative integers, and $m_l\le d$ for all $l=1,\dots,n$. Moreover, $(-K_S)\cdot c=3d-\sum_{l=1}^n m_l=1$. This class of $-1$-curve corresponds to the halfspace $dx_0+\sum_{l=1}^nm_lx_l\ge0$. We would like to show that this inequality is a consequence of the previous ones $x_0+x_{i_1}+x_{i_2}+x_{i_3}\ge0$ and $x_0+x_{j_1}+x_{j_2}\ge0$. That fact is proved by the claim below.

\end{proof}

\begin{claim}
Let $d,m_1,\dots,m_n$ be nonnegative integers, such that $d\ge1$, $m_l\le d$ for all $l=1,\dots,n$ and $3d-\sum_{l=1}^n m_l=1$. Then, the inequality $dx_0+\sum_{l=1}^nm_lx_l\ge0$ is the sum of $d-1$ inequalities of the form $x_0+x_{i_1}+x_{i_2}+x_{i_3}\ge0$, $1\le i_1<i_2<i_3\le n$, and $1$ inequality of the form $x_0+x_{j_1}+x_{j_2}\ge0$, $1\le j_1<j_2\le n$.
\end{claim}

\begin{proof}
If $d=1$, each $m_l\le1$ and there are exactly $2$ of them different from zero so that $3d-\sum_{l=1}^n m_l=1$. In this case, we already have that the inequality is of the form $x_0+x_{j_1}+x_{j_2}\ge0$. Assume by induction on $d$ that the claim is true whenever $1\le d<d_{max}$ for some $d_{max}>1$. Suppose that there are $m_1,\dots,m_n$ nonnegative integers such that $m_l\le d_{max}$ for all $l=1,\dots,n$ and $3d_{max}-\sum_{l=1}^n m_l=1$. In particular, $m_r>0$ for at least $3$ indices $r=i,j,k$ with largest $m_r$, otherwise $3d_{max}-1=m_i+m_j\le2d_{max}$ so $d_{max}\le1$. Therefore, the inequality $d_{max}x_0+\sum_{l=1}^nm_lx_l\ge0$ is the sum of $x_0+x_i+x_j+x_k\ge0$ and $d'x_0+\sum_{l=1}^nm'_lx_l$ with $d'=d_{max}-1$, $m'_r=m_r-1$ for $r=i,j,k$ and $m'_r=m_r$ otherwise.

On the other hand, the integers $d',m'_1,\dots,m'_n$ satisfy the same hypothesis of the claim: the sum $3d'-\sum_{l=1}^n m'_l$ remains the same and $m'_l\le d'$, otherwise there would be some $m'_r>d'=d_{max}-1$. Necessarily $m'_r=m_r=d_{max}$, which means that $r\neq i,j,k$, and $m_i,m_j,m_k\ge m_r=d_{max}$. However, we would have $1=3d_{max}-\sum_{l=1}^n m_l\le -d_{max}$, a contradiction. Therefore $d'x_0+\sum_{l=1}^nm'_lx_l$ is the sum of $d'-1=d_{max}-2$ inequalities $x_0+x_{i_1}+x_{i_2}+x_{i_3}\ge0$ and $1$ inequality $x_0+x_{j_1}+x_{j_2}\ge0$. Together with the inequality $x_0+x_i+x_j+x_k\ge0$ we are done.
\end{proof}

We conclude that the fundamental domain $\mathcal{P}$ for the action of $W$ in $\beta_n$ is a polytope given by the $n+1$ inequalities $x_0\ge -x_1-x_2-x_3$ and $x_1\le x_2\le \dots\le x_n\le0$. For $n<9$, $\mathcal{P}$ has finite volume simply because it is a subset of the finite volume polytope $\beta_n$. Much more interestingly, for $n=9$, we will see that $\mathcal{P}_9$ also has finite volume, even if $\beta_n$ does not. Moreover, the number of facets, is the minimum in order to have finite volume, so it is actually a simplex in $\H^9$, the simplest possible polytope. Even more, it is a Coxeter simplex, one among the only three Coxeter simplices that exist in $\H^9$.

\section{The case $n=9$}

Recall that the light cone is defined as $$\L=\L_n:=\{v\in\R^{1,n}\mid v^2\ge0\}\subset\R^{1,n}$$ and each point of $\H^n$ corresponds to exactly one line in the light cone, namely, the line connecting that point to the origin. In this sense, $$\H^n\cong\P(\L_n^{\circ})=\frac{\L_n^{\circ}}{\sim},$$ with $\L_n^{\circ}=\{v\in\R^{1,n}\mid v^2>0\}$ and the equivalence relation given by nonzero scalar multiplication. It is also a convention to interpret $\partial\H^n$ as the lines through the origin that are contained in $\partial\L_n=\{v\in\R^{1,n}\mid v^2=0\}$: $$\partial\H^n\cong\P(\partial\L_n)=\frac{\partial\L_n\backslash\{(0,\dots,0)\}}{\sim}.$$ Moreover the compactification of $\H^n$ is given by $$\overline{\H^n}=\H^n\cup\partial\H^n\cong\P(\L_n)=\frac{\L_n\backslash\{(0,\dots,0)\}}{\sim}.$$ 

We use coordinates $(x_0,x_1,\dots,x_n)$ to describe a vector in $\R^{1,n}$, and $(x_0:x_1:\dots:x_n)$ to describe the point in $\overline{\H^n}$ that corresponds to the line through the origin and the point $(x_0,x_1,\dots,x_n)\in\L_n\backslash\{(0,0,\dots,0)\}$.

\begin{prop}
\label{poliedral}
    The polytope $\mathcal{P}_9\subset\H^9$ given by the $10$ inequalities $x_0\ge -x_1-x_2-x_3$ and $x_1\le x_2\le \dots\le x_9\le0$ has finite volume. That is, its closure in $\overline{\H^9}$ contains finitely many points on the boundary $\partial\H^9$, precisely $2$, which are $(1:-1:0:0:0:0:0:0:0:0)$ and $\ (3:-1:-1:-1:-1:-1:-1:-1:-1:-1)$.
\end{prop}

\begin{proof}
    The set $\mathcal{P}_9$ is the intersection of the cone $\mathcal{C}\subset\R^{1,9}$ given by the inequalities $x_0\ge -x_1-x_2-x_3$ and $x_1\le x_2\le \dots\le x_9\le0$ with the hyperboloid sheet $$\H^9:=\{v\in\R^{1,9}\mid v^2=1,\ v\cdot e_0>0\}\subset\R^{1,9}.$$ By Definition \ref{simplex}, $\mathcal{P}_9$ has finite volume if and only if the cone $\mathcal{C}$ is contained in the light cone $$\L_n(S):=\{v\in\NS(S)\mid v^2\ge0\}\subset\NS(S).$$

For any point $v=(x_0,\dots,x_9)\in\mathcal{C}$ we have $x_0\ge -x_1-x_2-x_3\ge 0$, which implies that 
\begin{align*}
    x_0^2\ge &\ x_1^2+x_2^2+x_3^2+2x_1x_3+2x_2x_3+2x_1x_2\\ \ge &\ x_1^2+x_2^2+x_3^2+x_4^2+x_5^2+x_6^2+x_7^2+x_8^2+x_9^2
\end{align*}
because of the inequalities 
\begin{align*}
    x_1x_3\ge & \ x_4^2,\ x_5^2\\
    x_2x_3\ge & \ x_6^2,\ x_7^2\\
    x_1x_2\ge & \ x_8^2,\ x_9^2.
\end{align*} It follows that $v^2=x_0^2-\sum_{i=1}^9x_i^2\ge 0$, and so $\mathcal{C}\subset\mathcal{L}$.  

To find the points of $\mathcal{P}_9$ on the boundary of $\H^9$, we need to verify when the equality $v^2=0$ holds. It is necessary that all inequalities above are equalities. In particular, $x_1x_2=x_9^2$ only if one of the following two conditions holds: $x_2=x_9=0$, or $x_1=x_2=x_9$. It follows that the only $2$ points of $\overline{\mathcal{P}_9}$ on the boundary $\partial\H^9$ are $(1:-1:0:0:0:0:0:0:0:0)$ and $ (3:-1:-1:-1:-1:-1:-1:-1:-1:-1).$
\end{proof}

As a consequence, $\mathcal{P}_9$ is a simplex with the following $10$ vertices:
\begin{align*}
(1:0:0:0:0:0:0:0:0:0),&\\(2:-1:-1:0:0:0:0:0:0:0),&\\
(3:-1:-1:-1:0:0:0:0:0:0),&\\
(3:-1:-1:-1:-1:0:0:0:0:0),&\\(3:-1:-1:-1:-1:-1:0:0:0:0),&\\(3:-1:-1:-1:-1:-1:-1:0:0:0), &\\(3:-1:-1:-1:-1:-1:-1:-1:0:0),&\\
(3:-1:-1:-1:-1:-1:-1:-1:-1:0)&\in\H^9,\\
(1:-1:0:0:0:0:0:0:0:0),&\\(3:-1:-1:-1:-1:-1:-1:-1:-1:-1)&\in\partial\H^9,
\end{align*}
each of which is the unique intersection of $9$ hyperplanes delimiting $\mathcal{P}_9$.

Geometrically, each of these vertices corresponds to a ray in $\R^{1,9}$ spanned  by the class of a curve. These are the strict transforms of: a line that does not pass through any of the $9$ points $P_i$, a line passing through only $P_1$, a conic passing through $P_1$ and $P_2$, and cubics passing through $P_1$, $P_2,\ \dots, P_i$ only
for any $i=3,\dots,9$. Then we are ready to prove Theorem \ref{main theo n=9 part 1}, rephrased as follows:

\begin{theorem}
\label{main theo n=9 rephrased}
The nef cone of $\P^2_9$ admits a rational polyhedral fundamental cone with respect to the Cremona action of $W$. Namely, the cone $\mathcal{C}$ in $\R^{1,9}$ given by the inequalities $x_0\ge -x_1-x_2-x_3$ and $x_1\le x_2\le \dots\le x_9\le0$, or, in other words, the closed convex cone generated by the $10$ classes of curves mentioned above.
\end{theorem}
\begin{proof}
Since $W$ acts linearly on $\NS(S)$, by Lemma \ref{lema} and Theorem \ref{o dominio fundamental}, $\mathcal{C}\cap\L_9^{\circ}$ is a fundamental domain for $\Nef(S)\cap\L_9^{\circ}$. It remains to check that the nef classes on the boundary $\partial\L_9$ are in the orbit of those in $\mathcal{C}\cap\partial\L_9$. 

By Proposition \ref{poliedral}, the rays on $\mathcal{C}\cap\partial\L_9$ are the ones spanned  by the vectors $(1,-1,0,0,0,0,0,0,0,0)$ and $(3,-1,-1,-1,-1,-1,-1,-1,-1,-1)=-K_S$. These vectors are both nef, so $\mathcal{C}\subset\Nef(S)\implies\bigcup_{w\in W}w\mathcal{C}\subset\Nef(S)$.

Moreover, recall that the nef elements in $\partial\L_9$ are spanned by the elements in the orbit 
$W(1,-1,0,0,0,0,0,0,0,0)$ and $-K_S$. Therefore, $$\bigcup_{w\in W}w\mathcal{C}=\Nef(S).$$
\end{proof}

The Cartan Matrix $C_{\mathcal{P}_9}=[a_{ij}=-v_i\cdot v_j]$ reveals more. To calculate it, we recall that the normal vectors associated to the halfspaces are $v_0=e_0-e_1-e_2-e_3$, $v_i=e_{i}-e_{i+1}$ for $i=1,\dots,8$ and  $v_9=\sqrt{2}e_9$. Here it is: 

\begin{equation*}
C_{\mathcal{P}_9}=\begin{pmatrix}
2 & 0 & 0 & -1 & 0 &  0 & 0 & 0 & 0 & 0\\
0 & 2 & -1 & 0 & 0 &  0 & 0 & 0 & 0 & 0\\
0 & -1 & 2 & -1 & 0 &  0 & 0 & 0 & 0 & 0\\
-1 & 0 & -1 & 2 & -1 &  0 & 0 & 0 & 0 & 0\\
0 & 0 & 0 & -1 & 2 & -1 & 0 & 0 & 0 & 0\\
0 & 0 & 0 & 0 & -1 &  2 & -1 & 0 & 0 & 0\\
 0 & 0 & 0 & 0 & 0 &  -1 & 2 & -1 & 0 & 0 \\
0 & 0 & 0 & 0 & 0 & 0 & -1 & 2 & -1 & 0 \\
0 & 0 & 0 & 0 & 0  & 0 & 0  & -1 & 2 & -\sqrt{2} \\
0 & 0 & 0 & 0 & 0 & 0 & 0 & 0 & -\sqrt{2} & 2  \\
\end{pmatrix}.
\end{equation*}

The entries $2$, $0$, $-1$ e $-\sqrt{2}$ indicate that the corresponding pairs of hyperplanes make an angle of $\pi$, $\frac\pi2$, $\frac\pi3$ and $\frac\pi4$, respectively. Therefore, besides being a fundamental domain of $\beta_n$, we can obtain Theorem 2: this simplex is one of the three hyperbolic Coxeter 9-simplices up to congruence, according to \cite{GeometryII} part II chapter $5$ section 2.3 Table $4$. Its Coxeter diagram is 

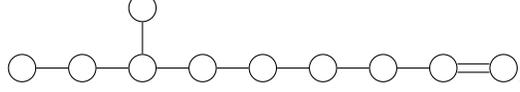
\begin{center}
\usetikzlibrary{positioning}
\begin{tikzpicture}
  [scale=.8,auto=left,every node/.style={circle,draw=black,fill=none}]
  
  \node (n1) at (1,10) {};
  \node (n2) at (2,10) {};
  \node (n3) at (3,10) {};
  \node (n4) at (4,10) {};
  \node (n5) at (5,10) {};
  \node (n6) at (6,10) {};
  \node (n7) at (7,10) {};
  \node (n8) at (8,10) {};
  \node (n9) at (9,10) {};
  \node (n10) at (3,11) {};  

  \foreach \from/\to in {n1/n2,n2/n3,n3/n4,n4/n5,n5/n6,n6/n7,n7/n8,n3/n10}
    \draw (\from) -- (\to);
    
  \draw ([yshift=-2pt]n8.east) -- ([yshift=-2pt]n9.west);
  \draw ([yshift=2pt]n8.east) -- ([yshift=2pt]n9.west);

\end{tikzpicture}.
\end{center}

Coxeter simplices are exactly the simplices that can tile the hyperbolic space by reflections. Not only $\mathcal{P}_9$ is a fundamental domain of $\beta_n$ with respect to the action of $W$, but also, if we add as a generator the reflection $r_9$ with respect to the hyperplane $(x_9=0)$, then $\mathcal{P}_9$ is a fundamental domain of $\H^9$ with respect to the action of $\left<W,r_9\right>$, which is also discrete. This completes the proof of Theorem \ref{main theo n=9 part 2}.

In dimension $n\ge10$, there are no Coxeter simplices in $\H^n$. In other words, if a polytope is a fundamental domain for the action of $W$ on the hyperbolic space of dimension $n\ge10$, then it is not a simplex. As we will see in the next section, the polytope we choose to be a fundamental domain for $\beta_n$ is no longer a simplex.

\section {The case $n\ge10$}

For $n\ge10$, the intersection of $\overline{\mathcal{P}}=\overline{\mathcal{P}_n}$ with the boundary $\partial\H^n$ is infinite, and therefore $\mathcal{P}_n$ has infinite volume. Recall that the shape of the nef cone is different on the two sides of the hyperplane $\{u\in\R^{1,n}\mid u\cdot K_S=0\}$ in $\R^{1,n}$. On the negative side of the nef cone, which is the only side on which we hope to find a polyhedral fundamental cone, we observe that $$\Nef^\text{e}(S_n)_{K_{S_n}\le0}\cap\H^n=\beta\cap\left(-\sqrt{\frac{2}{n-9}}K_S\right)^{\ge0},$$ by Remark \ref{lemma4.1_Tommaso}.

It is clear that the polytope $$\mathcal{P}^-_n:=\mathcal{P}_n\cap\left(-\sqrt{\frac{2}{n-9}}K_S\right)^{\ge0}$$ is a fundamental domain of $\beta\cap\left(-\sqrt{\frac{2}{n-9}}K_S\right)^{\ge0}$ with respect to the action of $W$. The reason is the same as in Lemma \ref{lema}, since $\mathcal{P}_n$ is a fundamental domain of $\beta_n$ with respect to the action of $W$ and the polytope $\beta\cap\left(-\sqrt{\frac{2}{n-9}}K_S\right)^{\ge0}$ is preserved by $W$.


\begin{prop}
\label{caso grande}
    For $n\ge10$, the polytope $\mathcal{P}^-_n$ has finite volume. Its closure in $\overline{\H^n}$ has exactly $2$ points on the boundary $\partial{\H^n}$: $(1:-1:0:\dots:0)$ and $(3:\underbrace{-1:\dots:-1}_{9\text{ coordinates }-1}:0:\dots:0)$.
\end{prop}

\begin{proof}
We must prove that the cone $\mathcal{C}_n$ in $\R^{1,n}$ given by the inequalities $x_0\ge -x_1-x_2-x_3$, $x_1\le x_2\le\dots\le x_n\le0$ and $3x_0\ge-\sum_{i=1}^{n}x_i$ is contained in the light cone $\L_n$. The proof is by induction on $n$ with base case $n=9$. When $n=9$, the last inequality is derived from the others, so $\mathcal{P}^-_9=\mathcal{P}_9$, and, by Proposition \ref{poliedral}, we know that it has finite volume and the exact $2$ mentioned points on $\partial\H^9$. 

Now suppose that $\mathcal{C}_{n-1}\subset\L_{n-1}$ for some $n\ge10$, and the rays spanned  by $(1,-1,0,\dots,0)$, $(3,\underbrace{-1,\dots,-1}_{9\text{ coordinates }-1},0,\dots,0)\in\R^{1,n}$ are the only rays of $\mathcal{C}_{n-1}$ on the boundary $\partial\L_{n-1}$. We will prove that the same is true for $n$. 

Let $v=(x_0,x_1,\dots,x_{n})\in\mathcal{C}_n$. If $x_{n}=0$, then the projection of $v$ in $\R^{1,n-1}$ given by $\overline{v}=(x_0, x_1, \dots, x_{n-1}) \in \R^{1,n-1}$ satisfies the defining inequalities of $\mathcal{C}_{n-1}.$ By induction, $\overline{v}\in\L_{n-1}$ and \begin{equation}\label{igualdade}x_0^2\ge\sum_{i=1}^{n-1}x_i^2=\sum_{i=1}^{n}x_i^2\end{equation} so $v\in\L_n$. Moreover, by induction, equality holds in (\ref{igualdade}). This is equivalent to $\overline{v}=\lambda(1,-1,0,\dots,0)$ or $\overline{v}=\lambda(3,\underbrace{-1,\dots,-1}_{9\text{ coordinates }-1},0,\dots,0)$ for some $\lambda>0$, if and only if $v=\lambda(1,-1,0,\dots,0)$ or $v=\lambda(3,\underbrace{-1,\dots,-1}_{9\text{ coordinates }-1},0,\dots,0)$ for some $\lambda>0$.

It remains to prove that if $x_{n}<0$, then $x_0^2-\sum_{i=1}^{n}x_i^2>0$. We may assume that $x_3=x_4$, otherwise we consider a small perturbation $v'=v+\epsilon(e_n-e_4)$ of $v$. We still have $v'=(x_0',x_1',\dots,x_n')\in\mathcal{C}_n$, but the value of $x_0^2-\sum_{i=1}^{n}x_i^2$ strictly decreases. We can increase $\epsilon$ until $x_{n}'=0$, and we fall on the previous case, or until $x_4'=x_3'$. 

Repeating the argument, we may assume that $x_4=x_5$, and so on. So we can suppose $x_3=x_4=\dots=x_{n-1}$ while $x_{n}<0$. The same argument shows that we may also assume that $x_2=x_3$, otherwise we consider perturbation $v'=v+\epsilon(e_2-e_1)$ of $v$.

We can rescale $v$ so that $x_0=1$, since $x_0>0$. This way we want to prove that $x_1^2+(n-2)x_2^2+x_{n}^2<1$ whenever we have $x_1+2x_2\ge-1$, $x_1\le x_2\le x_{n}<0$ and $x_1+(n-2)x_2+x_{n}\ge-3$. Let's define the region in $\R^3$ given by
\begin{equation}
    R=\left\{(x_1,x_2,x_n)\in\R^3;
    \begin{split}
        x_1+2x_2\ge-1,\\ x_1\le x_2\le x_{n}\le0,\\ x_1+(n-2)x_2+x_{n}\ge-3
    \end{split} 
    \right\}.
\end{equation}
We want to determine the maximal value of the convex function $f:R\to\R$ given by $f(x_1,x_2,x_n)=x_1^2+(n-2)x_2^2+x_{n}^2$. We know it must be achieved on some vertex of $R$.

Any $3$ of the $5$ facets of $R$ intersect in a single point. If this point verifies the other two inequalities, then it is a vertex of $R$, otherwise it is outside $R$. We check this for each such point, as well as the corresponding value of $f(x_1,x_2,x_n)$ when it is a vertex of $R$ in the table below.

\begin{center}
\begin{tabular}{|c c c|}

 \hline
 $(x_1,x_2,x_n)$ & is it a vertex of $R$? & $f(x_1,x_2,x_n)$ \\ [1ex] 
 \hline\hline
 $(0,0,0)$ & yes & $0$ \\[1ex]  
 \hline
 $(-1,0,0)$ & yes & $1$ \\[1ex] 
 \hline
 $(-3,0,0)$ & no &  \\[1ex] 
 \hline
 $\left(-\frac13,-\frac13,0\right)$ & $\text{yes}\iff n=10$ & $1$ \ if $n=10$ \\[1ex] 
 \hline
 $\left(-\frac{3}{n-1},-\frac{3}{n-1},0\right)$ & yes & $\frac{9}{n-1}$ \\[1ex] 
 \hline
 $\left(-\frac13,-\frac13,-\frac13\right) $ & no & \\ [1ex] 
 \hline
 $\left(-\frac3n,-\frac3n,-\frac3n\right)$ & yes & $\frac{9}{n}$ \\[1ex] 
 \hline
 $\left(-\frac13,-\frac13,\frac{n-10}{3}\right)$ & $\text{yes}\iff n=10$ & $1$ if $n=10$ \\ [1ex] 
 \hline
 $\left(-\frac{n-7}{n-3},-\frac{2}{n-3},-\frac{2}{n-3}\right)$ & yes & $\frac{(n-7)^2+4(n-1)}{(n-3)^2}$ \\[1ex] 
 \hline
 $\left(-\frac{n-8}{n-4},-\frac{2}{n-4},0\right)$ & yes & $\frac{(n-8)^2+4(n-2)}{(n-4)^2}$\\
 [1ex] 
 \hline
\end{tabular}
\end{center}

The points $(-3,0,0)$, $\left(-\frac13,-\frac13,-\frac13\right)$ do not satisfy the first and the last inequalities, respectively. The points $\left(-\frac13,-\frac13,0\right)$ and $\left(-\frac13,-\frac13,\frac{n-10}{3}\right)$ are vertices if and only if $n=10$, and the value of $f$ is then calculated taking $n=10$. One can check that for $n\ge10$, all these values are at most $1$, and whenever $x_n<0$, this value is less than $1$, which ends the proof.
\end{proof}

We can check that the vertices of the polytope $\mathcal{P}^-_n$ are the following $9n-71$ points:

\begin{align*}
    (1:0:\dots:0),&\\
    (1:-1:0:\dots:0),&\\
    (2:-1:-1:0:\dots:0),&\\
    (3:\underbrace{-1:\dots:-1}_{k \text{ coordinates}}:0:\dots:0),&\ 3\le k\le9 \\
    (m:\underbrace{-3:\dots:-3}_{m \text{ coordinates}}:0\dots:0),&\ 10\le m\le n\\
    (b-2:-(b-6):\underbrace{-2:\dots:-2}_{b \text{ coordinates}}:0:\dots:0),&\ 9\le b\le n-1\\
    (2b-2:-(b-3):-(b-3):\underbrace{-4:\dots:-4}_{b \text{ coordinates}}:0:\dots:0),&\ 8\le b\le n-2\\
    (3b:\underbrace{-b:\dots:-b}_{a \text{ coordinates}}:\underbrace{-(9-a):\dots:-(9-a)}_{b \text{ coordinates}}:0:\dots:0),&\ 3\le a\le8,\ 10\le a+b\le n.
\end{align*}

Finally, we conclude with the proof of Theorem \ref{main theo n>9}, rephrased below:

\begin{theorem} For $n\ge10$, the cone $\Nef^\text{e}(S_n)_{K_{S_n}\le0}$ admits a rational polyhedral fundamental cone with respect to the Cremona action of $W_n$. Namely, the cone $\mathcal{C}_n$ in $\R^{1,n}$ given by the inequalities $x_0\ge -x_1-x_2-x_3$, $x_1\le x_2\le \dots\le x_n\le0$ and $3x_0\ge -\sum_{i=1}^n x_i$.
\end{theorem}

\begin{proof} As in the proof of Theorem \ref{main theo n=9 rephrased}, we need to check that the union of the orbits of the rays spanned  by $(1,-1,0,\dots,0)$, $(3,\underbrace{-1,\dots,-1}_{9\text{ coordinates }-1},0,\dots,0)$ $\in\R^{1,n}$ coincide with $\Nef^\text{e}(S)_{K_S\le0}\cap\partial\L_n$. It is clear that these rays are contained in $\Nef^\text{e}(S)_{K_S\le0}\cap\partial\L_n$, and it remains to check the opposite inclusion.

Let $v=(x_0,x_1,\dots,x_n)\in\R^{1,n}$ be a point in a rational ray of $\Nef(S)\cap\partial\L_n$ such that $v\cdot K_S\le 0$. In particular, $x_0\ge0$. 
We may assume $v$ is 
on the lattice, that is, $x_0,x_1,\dots,x_n\in\Z$
. We may assume also, that $x_1\le x_2\le\dots\le x_n$, otherwise we replace $v$ with some other point of the same $W$-orbit satisfying this condition. Moreover, $x_n\le0$ and $3x_0\ge-\sum_{i=1}^n x_i$.

If $x_0+x_1+x_2+x_3\ge0$, $v$ lies on cone $\mathcal{C}_n$, then by Proposition \ref{caso grande}, it lies on one of the two desired rays. Otherwise, $x_0+x_1+x_2+x_3<0$ and we apply the element $w_0=\varphi_{123}\in W$ of Definition \ref{cremona}, to replace $v$ with $v'=(2x_0+x_1+x_2+x_3,-x_0-x_2-x_3,-x_0-x_1-x_3,-x_0-x_1-x_2,x_4,\dots,x_n)$, and now the new first coordinate $x_0'<x_0$ is strictly smaller than before. We rearrange the new coordinates in order to have $x_1'\le x_2'\le\dots\le x_n'$ and we can repeat this process finitely many times, since $x_0$ is a positive integer. It only stops when $v\in\mathcal{C}_n$ and we are done.
\end{proof}


\begin{prop}
    The polytope $\mathcal{P}^-$ is a Coxeter polytope if and only if $n=10,11,13$.
\end{prop}

\begin{proof}
    Let's calculate the Cartan Matrix of $\mathcal{P}^-$ taking into consideration its normal vectors. As in the case of $C_{\mathcal{P}_9}$, we have the normal vectors $v_0=e_0-e_1-e_2-e_3$, $v_i=e_{i}-e_{i+1}$ for $i=1,\dots,n-1$, $v_n=\sqrt{2}e_n$, and finaly, the new normal vector $v_{n+1}=-\sqrt{\frac{2}{n-9}}K_S$. The Cartan Matrix is therefore 

\begin{equation*}
C_{\mathcal{P}^-}=
\begin{pmatrix}
2 & 0 & 0 & -1 & 0 &  \dots & 0 & 0 & 0\\
0 & 2 & -1 & 0 & 0 &  \dots & 0 & 0 & 0\\
0 & -1 & 2 & -1 & 0 &  \ddots & \vdots & \vdots & \vdots\\
-1 & 0 & -1 & 2 & -1 &  \ddots & 0 & 0 & 0\\
0 & 0 & 0 & -1 & 2 &  \ddots & 0 & 0 & 0\\
 \vdots & \vdots  &  \ddots & \ddots  & \ddots  & \ddots & -1 & 0 & 0 \\
0 & 0 & \dots & 0 & 0 & -1 & 2 & -\sqrt2  & 0\\
0 & 0 & \dots & 0  & 0 & 0  & -\sqrt2 & 2  & \frac{-2}{\sqrt{n-9}}\\
0 & 0 & \dots & 0  & 0 & 0  & 0 & \frac{-2}{\sqrt{n-9}}  & 2\\
\end{pmatrix}.
\end{equation*}

The new entries $ \frac{-2}{\sqrt{n-9}}$ indicate that $\theta_{n,n+1}=\cos^{-1}\left(\frac{1}{\sqrt{n-9}}\right)$, which is a submultiple of $\pi$ if and only if $n=10,11,13$. In these cases, we have $\theta_{n,n+1}=0,\frac{\pi}{4},\frac{\pi}{3}$, respectively, and the Coxeter diagrams are respectively 

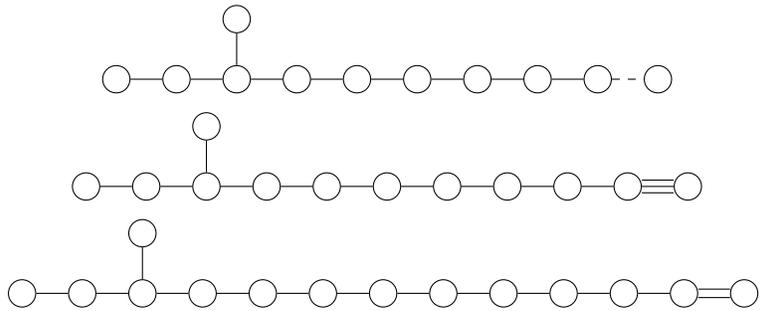
\begin{center}
\begin{tikzpicture}
  [scale=.8,auto=left,every node/.style={circle,draw=black,fill=none}]
  
  \node (n1) at (1,10) {};
  \node (n2) at (2,10) {};
  \node (n3) at (3,10) {};
  \node (n4) at (4,10) {};
  \node (n5) at (5,10) {};
  \node (n6) at (6,10) {};
  \node (n7) at (7,10) {};
  \node (n8) at (8,10) {};
  \node (n9) at (9,10) {};
    \node (n10) at (10,10) {};
  \node (n11) at (3,11) {};  

  \foreach \from/\to in {n1/n2,n2/n3,n3/n4,n4/n5,n5/n6,n6/n7,n7/n8,n8/n9,n3/n11}
    \draw (\from) -- (\to);
\draw[dashed](n9) -- (n10);
\end{tikzpicture}
\end{center}

\begin{center}
\begin{tikzpicture}
  [scale=.8,auto=left,every node/.style={circle,draw=black,fill=none}]
  
  \node (n1) at (1,10) {};
  \node (n2) at (2,10) {};
  \node (n3) at (3,10) {};
  \node (n4) at (4,10) {};
  \node (n5) at (5,10) {};
  \node (n6) at (6,10) {};
  \node (n7) at (7,10) {};
  \node (n8) at (8,10) {};
  \node (n9) at (9,10) {};
    \node (n10) at (10,10) {};
    \node (n11) at (11,10) {};
  \node (n12) at (3,11) {};  

  \foreach \from/\to in {n1/n2,n2/n3,n3/n4,n4/n5,n5/n6,n6/n7,n7/n8,n8/n9,n9/n10,n3/n12}
    \draw (\from) -- (\to);
\draw
([yshift=-3pt]n10.east) -- ([yshift=-3pt]n11.west);
\draw
([yshift=0pt]n10.east) -- ([yshift=0pt]n11.west);
  \draw ([yshift=3pt]n10.east) -- ([yshift=3pt]n11.west);
\end{tikzpicture}
\end{center}

\begin{center}
\begin{tikzpicture}
  [scale=.8,auto=left,every node/.style={circle,draw=black,fill=none}]
  
  \node (n1) at (1,10) {};
  \node (n2) at (2,10) {};
  \node (n3) at (3,10) {};
  \node (n4) at (4,10) {};
  \node (n5) at (5,10) {};
  \node (n6) at (6,10) {};
  \node (n7) at (7,10) {};
  \node (n8) at (8,10) {};
  \node (n9) at (9,10) {};
    \node (n10) at (10,10) {};
    \node (n11) at (11,10) {};
    \node (n12) at (12,10) {};
    \node (n13) at (13,10) {};
  \node (n14) at (3,11) {};  

  \foreach \from/\to in {n1/n2,n2/n3,n3/n4,n4/n5,n5/n6,n6/n7,n7/n8,n8/n9,n9/n10,n10/n11,n11/n12,n3/n14}
    \draw (\from) -- (\to);
\draw
([yshift=-2pt]n12.east) -- ([yshift=-2pt]n13.west);
  \draw ([yshift=2pt]n12.east) -- ([yshift=2pt]n13.west);
\end{tikzpicture}.
\end{center}

Since the value $\frac{1}{\sqrt{n-9}}$ decreases to $0$ with $n$, it is not possible to find another submultiple of $\pi$, since it should be between $\frac{\pi}{3}$ and $\frac{\pi}{2}$. 
\end{proof}

\paragraph{\textbf{Acknowledgements}} I would like to thank Carolina Araujo for the problem that originated this paper and all her support throughout this work. I would also like to thank Mikhail Belolipetsky for the insightful discussions about reflection groups. Furthermore, I would like to thank CAPES (Coordernação de Aperfeiçoamento de Pessoal de Nível Superior), CNPq (Conselho Nacional de Desenvolvimento Científico e Tecnológico) and FAPERJ (Fundação Carlos Chagas Filho de Amparo à Pesquisa do Estado do Rio de Janeiro) for the financial support. 

 \bibliographystyle{alphaurl}
\bibliography{main}
\end{document}